\documentclass[a4paper]{amsart}

\usepackage{amsmath,amssymb,amsthm,units, stmaryrd,stackrel,relsize,bm}

\usepackage{yfonts,upgreek,units,stmaryrd,color,graphicx}
\usepackage[all,cmtip]{xy}
\usepackage{accents}
\usepackage{enumerate}

\usepackage[textsize=footnotesize,color=yellow, bordercolor=white]{todonotes}

\newtheorem{theorem}{Theorem}[section]

\newtheorem{lemma}[theorem]{Lemma}

\newcounter{cl}[theorem]
\newtheorem{claim}[cl]{Claim}
\newtheorem{subclaim}{Subclaim}[cl]
\newtheorem*{subclaim*}{Subclaim}
\newtheorem{corollary}[theorem]{Corollary}

\theoremstyle{definition}
\newtheorem{definition}[theorem]{Definition}
\newtheorem*{definition*}{Definition}

\theoremstyle{remark}

\newtheorem*{remark*}{Remark}

\DeclareSymbolFont{AMSb}{U}{msb}{m}{n}
\DeclareMathSymbol{\N}{\mathbin}{AMSb}{"4E}
\DeclareMathSymbol{\Z}{\mathbin}{AMSb}{"5A}
\DeclareMathSymbol{\R}{\mathbin}{AMSb}{"52}
\DeclareMathSymbol{\Q}{\mathbin}{AMSb}{"51}
\DeclareMathSymbol{\I}{\mathbin}{AMSb}{"49}
\DeclareMathSymbol{\C}{\mathbin}{AMSb}{"43}

\usepackage[
    colorlinks=true, 
    citecolor=red,
    linkcolor=blue,
    urlcolor=blue]{hyperref}
\newcommand\ad{{\mathsf{AD}}}
\newcommand\ac{{\mathsf{AC}}}

\newcommand\zf{{\mathsf{ZF}}}
\newcommand\dc{{\mathsf{DC}}}

\newcommand\HOD{{\mathsf{HOD}}}
\newcommand\Ord{{\mathsf{Ord}}}

\newcommand{\BS}{{}^\omega\omega}

\newcommand{\cQ}{\mathcal{Q}}

\newcommand{\cM}{\mathcal{M}}
\newcommand{\cN}{\mathcal{N}}

\newcommand{\bR}{\mathbb{R}}

\newcommand{\langpm}{\mathcal{L}_{\mathrm{pm}}(\{\dot x_i \st i \in \omega\})} 
\newcommand{\lpm}{\mathcal{L}_{\mathrm{pm}}} 
\newcommand{\ctbleset}{\mathcal{P}_{\omega_1}(\mathbb{R})}

\newcommand{\st}{\colon}

\newcommand{\comm}[1]{}

\begin{document}
\title{The Axiom of Determinacy Implies Dependent Choices in Mice}

\subjclass[2010]{03E45, 03E60, 03E25, 03E15} 

\keywords{Infinite Game, Determinacy, Dependent Choices, Inner Model
  Theory, Mouse}

\author{Sandra M\"uller} \address{Sandra M\"uller, Kurt G\"odel
  Research Center, Institut f\"ur Mathematik, UZA 1, Universit\"at
  Wien. Augasse 2-6, 1090 Wien, Austria.}
\email{mueller.sandra@univie.ac.at}

\date{\today}

\begin{abstract}
  We show that the Axiom of Dependent Choices, $\dc$, holds in
  countably iterable, passive premice $\cM$ constructed over their reals
  which satisfy the Axiom of Determinacy, $\ad$, in a $\zf+\dc_{\bR^\cM}$
  background universe. This generalizes an argument of Kechris for
  $L(\bR)$ using Steel's analysis of scales in mice. In particular, we
  show that for any $n \leq \omega$ and any countable set of reals $A$
  so that $M_n(A) \cap \bR = A$ and $M_n(A) \vDash \ad$, we have that
  $M_n(A) \vDash \dc$. 
\end{abstract}
\maketitle
\setcounter{tocdepth}{1}

\section{Introduction}

We prove that in passive, countably iterable mice $\cM$ constructed
over their reals, $\ad$, the \emph{Axiom of Determinacy}, implies
$\dc$, the \emph{Axiom of Dependent Choices}, working in a background
universe which satisfies $\zf+\dc_{\bR^\cM}$. Here we write
$\bR^\cM = \bR \cap \cM$ for the set of reals in $\cM$.

Recall that $\dc$ is the following statement: For every nonempty set
$X$ and every binary relation $P$ on $X$,
\[ \forall a \in X \exists b \in X \; P(a,b) \, \Rightarrow \, \exists
  f \colon \omega \rightarrow X \, \forall n \; P(f(n), f(n+1)). \]
Moreover, $\dc_\bR$ denotes $\dc$ restricted to the case where
$X = \bR$ and more generally, for some nonempty set $Y$, $\dc_Y$
denotes $\dc$ restricted to the case where $X = Y$.

G\"odel's constructible universe over the reals $L(\bR)$ is the
closure of $\bR$ under the definable power set operation. Kechris
showed in \cite{Ke84} that in $L(\bR)$, the Axiom of Determinacy
implies the Axiom of Dependent Choices. His proof is based on the
analysis of scales in $L(\bR)$ which was developed by Martin,
Moschovakis, and Steel (see \cite{MMS82}, \cite{Mo08}, \cite{MaSt08},
and \cite{St08a}). A generalization of \cite{Ke84} and the analysis of
scales to the Dodd-Jensen core model over $\bR$ was shown by
Cunningham in \cite{Cu95}. We prove the following more general result
for arbitrary mice building on the analysis of scales in mice from
\cite{St08b}. Note that, in contrast to Kechris's result for $L(\bR)$,
our result requires $\dc_{\bR^\cM}$ to hold in $V$ in order to
consider countable elementary substructures of $\cM$. We will make it
clear in the proof where the countability of the model in question is
used.

\begin{theorem}[$\zf$]\label{thm:ADimpliesDC}
  Let $\cM$ be a passive, countably iterable $\bR^\cM$-premouse such
  that $\cM \vDash \ad$ and suppose that $\dc_{\bR^\cM}$ holds in
  $V$. Then $\cM \vDash \dc$.
\end{theorem}

For countable mice it is not necessary to assume
$\dc_{\bR^\cM}$, see Theorem \ref{thm:ADimpliesDCctblemice}. In
particular, Theorem \ref{thm:ADimpliesDC} holds for mice of the form
$M_n(A)$ for some $n \leq \omega$ and some countable set of reals $A$
such that $M_n(A) \cap \bR = A$. This result is for example used in
\cite{AgMu}, where the authors derive a model with $\omega+n$ Woodin
cardinals from a model of the form $M_n(A)$ with $M_n(A) \cap \bR = A$
which satisfies the Axiom of Determinacy.

Finally, we would like to thank the referee for the helpful comments
and suggestions.

\section{Countable mice in a $\zf$ background
  universe}\label{sec:dcinthebackground}

For simplicity, we first show the following version of Theorem
\ref{thm:ADimpliesDC} for countable mice and argue in the next section
that this implies Theorem \ref{thm:ADimpliesDC}. As mentioned above,
we do not require any form of choice in the background universe for
this result.

\begin{theorem}[$\zf$]\label{thm:ADimpliesDCctblemice}
  Let $\cM$ be a countable, passive, $(\omega_1+1)$-iterable
  $\bR^\cM$-premouse such that $\cM \vDash \ad$. Then
  $\cM \vDash \dc$.
\end{theorem}

For the definition of premice and $(\omega_1+1)$-iterability we refer
the reader to \cite{St10}, and to \cite{MS94} and \cite{SchStZe02} for
more background. Moreover, we refer to \cite{St08b} for the notion of
$X$-premice for arbitrary sets $X$. First, we recall the notion of
iterability we use in the statement of Theorem \ref{thm:ADimpliesDC}.

\begin{definition}
  Let $A$ be a set of reals and suppose $\cM$ is an $A$-premouse.  We
  say that $\cM$ is \emph{countably iterable} iff whenever $\bar \cM$
  is a countable $\bar A$-premouse for a set of reals $\bar A$ and
  there is an elementary embedding
  $\pi \colon \bar \cM \rightarrow \cM$, then $\bar \cM$ is
  $(\omega_1+1)$-iterable.
\end{definition}

To prove Theorem \ref{thm:ADimpliesDCctblemice}, we will show that the
argument in the proof of \cite[Theorem 4.1]{St08b} which yields the
existence of scales in $\cM$ using $\cM \vDash \dc$, can be used to
show the existence of quasi-scales without using $\dc$ in
$\cM$. Moreover, we sketch how we can adapt the argument from
\cite{Ke84} for $V = L(\bR)$ to obtain $\cM \vDash \dc$ from these
quasi-scales.

Following the notation in \cite{St08b}, we write $K(\bR)$ for the
model-theoretic union of all $\omega$-sound, countably iterable
premice over $\bR$ which project to $\bR$. Using $\dc_\bR$, it is easy
to show that any two such premice $\cM$ and $\cN$ line up,
i.e. satisfy $\cM \unlhd \cN$ or $\cN \unlhd \cM$. Therefore $K(\bR)$
is well-defined.

\begin{remark*} If we consider premice $\cM(\bR)$ constructed over all
  reals $\bR = \bR^V$, e.g. $\cM(\bR) = K(\bR)$ or
  $\cM(\bR) = M_1(\bR)$, it is easy to see that $\dc$ in $V$ (and in
  fact, using the argument at the beginning of the proof of Theorem
  \ref{thm:ADimpliesDCctblemice}, even $\dc_\bR$ in $V$) already
  implies $\dc$ in $\cM(\bR)$ as every function
  $f \colon \omega \rightarrow \bR$ witnessing $\dc$ in $V$ can be
  coded by a single real and is therefore already contained in
  $\cM(\bR)$. But the same does not hold in general for models $\cM$
  as in Theorem \ref{thm:ADimpliesDC} with $\bR^\cM \subsetneq \bR$
  since if $f \colon \omega \rightarrow \bR^\cM$ is a function
  witnessing $\dc$ in $V$ for reals in $\cM$ for some relation $P$, it
  can be coded by a single real in $V$, but this real need not be in
  $\bR^{\cM}$.
\end{remark*}

For the reader's convenience, we repeat parts of the arguments from
\cite{Ke84} and \cite{St08b} to point out the modifications we need to
make. We start by recalling the notions of quasi-norm and quasi-scale
which go back to \cite{Ke84}.

    \begin{definition}\label{def:quasi-norm}
      Let $B \subseteq \bR$. A relation $\leq$ on $B$ is a
      \emph{quasi-norm} iff
      \begin{enumerate}
      \item $\leq$ is a linear preordering on $B$, i.e. $\leq$ is
        reflexive, transitive, and for all $x,y \in B$, $x \leq y$ or
        $y \leq x$, and
      \item there is no infinite descending chain in $<$, where for
        $x,y \in B$, we write $x < y$ iff $x \leq y$ and
        $\neg(y \leq x)$. \label{it:quasi-norm2}
      \end{enumerate}
    \end{definition}

    \begin{definition}
      Let $B \subseteq \bR$. A \emph{quasi-scale} on $B$ is a sequence
      of quasi-norms $(\leq_i)_{i<\omega}$ on $B$ such that if
      $x_i \in B$ for $i < \omega$ with $x_i \rightarrow x$ as
      $i \rightarrow \infty$ and if for each $i$ there is some
      $n_i \in \omega$ such that $x_k \equiv_i x_{n_i}$ for all
      $k \geq n_i$, \footnote{We write $x \equiv_i y$ iff $x \leq_i y$
        and $y \leq_i x$.} then
      \begin{enumerate}
      \item $x \in B$ (\emph{limit property}), and
      \item for all $i < \omega$, $x \leq_i x_{n_i}$ (\emph{lower
          semi-continuity}).
      \end{enumerate}
    \end{definition}

    If we replace \eqref{it:quasi-norm2} in Definition
    \ref{def:quasi-norm} by ``every nonempty subset of $B$ has a
    $\leq$-least element'', we obtain the usual definitions of norm
    and scale. Hence, under $\dc_\bR$ every quasi-scale is a scale.

    We shall need the following lemma from \cite{Ke84} which is
    motivated by the proof of the Third Periodicity Theorem (see
    \cite[Theorem 6E.1]{Mo09}). Recall that $\ac_{\omega,\bR}$ denotes
    countable choice for reals, i.e., for all relations $P$ on
    $\omega \times \bR$,
    \[ \forall n \in \omega \exists r \in \bR \; P(n,r) \, \Rightarrow
      \, \exists f \colon \omega \rightarrow \bR \, \forall n \in
      \omega \; P(n,f(n)).\]

    \begin{lemma}[$\ac_{\omega, \bR}$] \label{lem:Kechris} Suppose $B$
      is a nonempty set of reals and $(\leq_i)_i$ is a quasi-scale on
      $B$. Let $\Gamma$ be a pointclass containing $B$ such that the
      relation
      \[ R(i,x,y) \Leftrightarrow (x,y \in B \wedge x \leq_i y) \] is
      in $\Gamma$. Moreover, suppose that $\Gamma$ is closed under
      recursive substitutions, $\neg$, $\wedge$, $\vee$, and
      existential and universal quantification over $\bR$. Then $B$
      contains a real $x$ such that
      $\{(n,m) \in \omega \times \omega \st x(n) = m\}$ is in
      $\Gamma$.
    \end{lemma}

Now we turn to the proof of Theorem \ref{thm:ADimpliesDCctblemice}.

\begin{proof}[Proof of Theorem \ref{thm:ADimpliesDCctblemice}] Work in
  the countable mouse $\cM$ and note that it suffices to prove
  $\dc_{\bR^\cM}$ since there is a definable surjection
  $F: \Ord^\cM \times \bR^\cM \twoheadrightarrow \cM$ (see Proposition
  2.4 in \cite{St08b}). Suppose $\dc_{\bR^\cM}$ fails, i.e. there is a
  relation $P \subseteq \bR^\cM \times \bR^\cM$ such that
  $\forall x \in \bR^\cM \exists y \in \bR^\cM \, P(x,y)$, but there
  is no $f \colon \omega \rightarrow \bR^\cM$ with $P(f(n), f(n+1))$
  for all $n \in \omega$.

  Let $\xi < \Ord^\cM$ be a large enough limit ordinal such that
  $P \in \cM | \xi$ and $\cM | \xi$ is passive. We may assume that
  such a limit ordinal exists because the general case when $\Ord^\cM$
  need not be a limit of limit ordinals can be shown similarly using
  the $S$-hierarchy (see the end of the proof of Theorem 2.1 in
  \cite{St08a}). Let $\alpha$ be the least ordinal below $\xi$ such
  that $\cM | \alpha \prec_1 \cM | \xi$ (in the sense of Definition
  4.4 in \cite{St08b}) and note that $\alpha$ is a limit ordinal. The
  statement
  \begin{align*}
    \exists P & \subseteq \bR^\cM \times \bR^\cM
    (\forall x \in \bR^\cM \exists y \in \bR^\cM \, P(x,y)) \; \wedge
    \\
    & \; \neg \exists f \colon \omega \rightarrow \bR^\cM \, \forall n \,
    P(f(n), f(n+1))
  \end{align*}
  is $\Sigma_1$ in the parameter $\bR^\cM$ as any
  $f \colon \omega \rightarrow \bR^\cM$ can be coded by a
  real. Therefore it follows that there is a counterexample to
  $\dc_{\bR^\cM}$ (in $\cM$) inside $\cM | \alpha$. To finish the
  proof, we use the following lemma.

  \begin{lemma}\label{lem:uniformization} 
    Every relation $P \subseteq \bR^\cM \times \bR^\cM$ in
    $\cM | \alpha$ can be uniformized in $\cM$, i.e. there is a
    function $F \colon \bR^\cM \rightarrow \bR^\cM$ in $\cM$ such that
    for all $x \in \bR^\cM$,
    \[ \exists y P(x,y) \Rightarrow P(x, F(x)). \]
  \end{lemma}

  Applying Lemma \ref{lem:uniformization} to the counterexample $P$
  above, we can define a function
  $f \colon \omega \rightarrow \bR^\cM$ by letting
  $f(0) = a \in \bR^\cM$ be arbitrary and $f(n+1) = F(f(n))$. Then
  $P(f(n), f(n+1))$ holds for all $n$, contradicting the choice of
  $P$. So it suffices to prove Lemma \ref{lem:uniformization}.

 \begin{proof}[Proof of Lemma \ref{lem:uniformization}]
   The proof divides into three claims. The first claim uses fine
   structural arguments to obtain definability for the sets of reals
   in $\cM|\alpha$. The key part of the argument is Claim
   \ref{cl:quasi-scale}, where we show the existence of
   quasi-scales. Finally, in Claim \ref{cl:basisresult} we piece Claim
   \ref{cl:quasi-scale} and Lemma \ref{lem:Kechris} together to obtain
   a basis result which will imply the existence of a uniformizing
   function, as desired.
   
  \begin{claim}\label{cl:definability}
    Every set of reals in $\cM | \alpha$ is $\Sigma_1$-definable in
    $\cM | \alpha$ from parameters in $\bR^\cM \cup \{ \bR^\cM \}$.
  \end{claim}
  \begin{proof}
    Standard fine structural arguments show that $\cM | \alpha$ has a
    $\Sigma_1$ Skolem function which is $\Sigma_1$ definable in
    $\cM | \alpha$ (without parameters). As in the proof of Lemma 1.11
    in \cite{St08a} for $L(\bR)$, this together with the fact the we
    chose $\alpha$ minimal with the property that
    $\cM | \alpha \prec_1 \cM | \xi$ yields that there is a partial
    surjection $h \colon \bR^\cM \twoheadrightarrow \cM | \alpha$ such
    that the graph of $h$ is $\Sigma_1$ definable in $\cM | \alpha$
    from parameter $\bR^\cM$. Hence, every set of reals in
    $\cM | \alpha$ is $\Sigma_1$ definable in $\cM | \alpha$ from
    parameters in $\bR^\cM \cup \{ \bR^\cM \}$, as desired.
  \end{proof}

    \begin{claim}\label{cl:quasi-scale}
      Let $B \subseteq \bR^\cM$ be a set of reals which is
      $\Sigma_1$-definable in $\cM|\alpha$ from some real parameter
      $r$ and the parameter $\bR^\cM$. Then there is a quasi-scale
      $(\leq_i)_{i<\omega}$ on $B$ which is also $\Sigma_1$-definable
      in $\cM|\alpha$ from the parameters $r$ and $\bR^\cM$.
    \end{claim}

    \begin{proof}
      Here we use Steel's analysis of scales in mice (see
      \cite{St08b}) under $\dc_\bR$ and observe that it can be used to
      obtain a quasi-scale without any use of $\dc_\bR$. In order to
      show how to do this, we sketch parts of his argument below.

      So let $B \subseteq \bR^\cM$ be a set of reals which is
      $\Sigma_1$-definable over $\cM|\alpha$ with some real parameter
      $r$ and parameter $\bR^\cM$. Hence for some $\Sigma_1$ formula
      $\varphi$,
      \[ x \in B \text{ iff } \cM | \alpha \vDash \varphi(x,r,
        \bR^\cM) \] for all $x \in \bR^\cM$. Recall that $\alpha$ is a
      limit ordinal. If $\cM | \alpha$ satisfies ``$\Theta$ exists'',
      let $\alpha^* = \Theta^{\cM|\alpha}$, otherwise let
      $\alpha^* = \alpha$. Now write for each $\beta < \alpha^*$ and
      $x \in \bR^\cM$,
      \[ x \in B^\beta \text{ iff } \cM | \beta \vDash
        \varphi(x,r,\bR^\cM). \] By \cite[Lemma 3.2]{St08b}, applied
      inside $\HOD_{x,\Sigma}$, where $x$ is a real coding
      $\cM | \alpha$ and $\Sigma$ is an iteration strategy for $\cM$,
      we obtain $B = \bigcup_{\beta < \alpha^*}B^\beta$. Note that
      $\Sigma$ is amenable to $\HOD_{x, \Sigma}$, so the (canonically
      well-ordered) fragment $\Sigma \cap \HOD_{x, \Sigma}$ is
      available within the model $\HOD_{x, \Sigma}$ and witnesses
      iterability there. Moreover, $\HOD_{x,\Sigma}$ is a model of the
      Axiom of Choice. Steel constructs in the proof of Theorem 4.1 in
      \cite{St08b} a closed game representation $x \mapsto G_x^\beta$
      of $B^\beta$ for each $\beta < \alpha^*$. We briefly sketch the
      argument here to show that it can be done in our situation as
      well. First, recall the definition of a closed game
      representation, which was essentially introduced in \cite{Mo08}.

      \begin{definition*}
        Let $x$ be a real and $G_x$ a closed game where Player I plays
        elements of $\BS \times \gamma$ for some ordinal $\gamma$ and
        Player II plays elements of $\BS$, and there is some relation
        $\cQ \subseteq (\omega^{{<}\omega})^{{<}\omega} \times
        \gamma^{{<}\omega}$ such that Player I wins the run
        $((x_0, \gamma_0), x_1, (x_2, \gamma_1), x_3, \dots)$ of $G_x$
        iff
        \[ \forall n \cQ( (x \upharpoonright n, x_0 \upharpoonright n,
          \dots, x_n \upharpoonright n), (\gamma_0, \dots, \gamma_n)
          ).  \] In particular, $G_x$ is continuously associated to
        $x$. We say $x \mapsto G_x$ is a \emph{closed game
          representation of $B$} iff $B$ is the set of all $x$ such
        that Player I has a winning quasi-strategy in $G_x$.
      \end{definition*}

      We now define a closed game representation $x \mapsto G_x^\beta$
      of $B^\beta$ for each $\beta < \alpha^*$. Fix $\beta < \alpha^*$
      and $x$. Let $G_x^\beta$ be the following game:
      
      \[ \begin{array}{c|cccccc} \mathrm{I} & i_0, x_0, \gamma_0 & &
                                                                     i_1,
                                                                     x_2,
           \gamma_1& &\hdots & \\ \hline
    \mathrm{II} & & x_1 & & x_3 & &\hdots 
         \end{array} \]
       
       The rules of the game ask Player I to play
       $i_0, i_1, \dots \in \{0,1\}$ in order to code a theory $T$ in
       the language $\langpm$ of premice with additional constant
       symbols $\{\dot x_i \st i \in \omega\}$ such that every model
       $\cN^*$ of $T$ is well-founded. Furthermore, the players
       alternate playing reals $x_i$, $i \in \omega$, and Player I
       plays additional ordinals $\gamma_i$, $i \in \omega$. The
       theory ensures that for every model $\cN^*$ of $T$, for all
       $i \in \omega$, $(\dot x_i)^{\cN^*} = x_i$ and the definable
       closure of $\{ x_i \st i \in \omega \}$ in
       $\cN^* \upharpoonright \lpm$ is an elementary submodel $\cN$ of
       $\cN^* \upharpoonright \lpm$. By considering its transitive
       collapse we can assume that $\cN$ is transitive. The winning
       conditions for Player I require that he plays the theory $T$
       such that
       \begin{align*}
         \cN \vDash \text{``}V = 
         K(\bR) + \varphi(&x,r,\bR) + \text{ all of my proper initial
           segments} \\ & \text{do not satisfy } \varphi(x,r,\bR)\text{''}.
       \end{align*}
       In addition, he is using the ordinals $\gamma_i$ to not only
       verify well-foundedness of $\cN$ by embedding the ordinals into
       $\omega\beta$, but also to verify iterability of $\cN$ by
       embedding the local $\HOD$'s of $\cN$ into the local $\HOD$'s
       of $\cM|\beta$. This latter embedding corresponds to the
       embedding of the ordinals. This amount of details suffices for
       our sketch of the argument, the formal definition of
       $G_x^\beta$ can be found in \cite[Section 4]{St08b}.

    Let
      \begin{align*}
        B_k^\beta(x,u) \Leftrightarrow & \; u \text{ is a position of
        length } k \text{ from which}\\ & \text{Player I has a winning
        quasi-strategy in } G_x^\beta.
      \end{align*}

      We aim to show that each $B_k^\beta$ is in $\cM | \alpha$ and
      that the map $(\beta, k) \mapsto B_k^\beta$ is $\Sigma_1$
      definable over $\cM|\alpha$ with parameters $r$ and
      $\bR^\cM$. In order to do that, we consider \emph{honest
        positions} in the game $G_x^\beta$, which are positions where
      Player I played the theory $T$ up to this point according to the
      theory of an initial segment $\cM|\xi$ of the true model
      $\cM|\beta$ and the embeddings induced by the ordinals
      $\gamma_i$ according to an elementary embedding between the
      local $\HOD$'s of $\cM|\xi$ and the local $\HOD$'s of the true
      model $\cM|\beta$.

      \begin{definition*}
        We say a position $u = ((i_n,x_{2n},\gamma_n,x_{2n+1}) \st n<k)$
        in the game $G_x^\beta$ is \emph{$(\beta,x)$-honest} iff
        $\cM | \beta \vDash \varphi(x,r,\bR)$ and if
        $\xi \leq \beta$ is least such that
        \[ \cM | \xi \vDash \varphi(x,r,\bR), \] then $x_0 = x$ if
        $k > 0$ and if $\cM^+|\xi$ denotes the canonical expansion of
        $\cM|\xi$ to the language $\lpm(\{x_i \st i < 2k\})$ by
        letting $(\dot x_i)^{\cM^+} = x_i$ for $i<2k$,
        \begin{enumerate}
        \item $\cM^+|\xi$ satisfies all sentences in $T$ determined up
          to the position $u$,
        \item the embedding given by $\cM^+|\xi$ and the ordinals
          $\gamma_i$ for $i<k$ is well-defined and can be extended to
          an order-preserving map
          \[ \pi \colon \omega\xi \rightarrow \omega\beta, \] and
        \item this embedding can be extended to an elementary
          embedding between the relevant local $\HOD$'s.
        \end{enumerate}
      \end{definition*}

      The formal definition of honest positions can be found in
      \cite[Section 4]{St08b}. We let $H_k^\beta(x,u)$ iff $u$ is a
      $(\beta,x)$-honest position of length $k$. The following
      subclaim concerning the definability of honest positions is the
      analogue of \cite[Claim 4.2]{St08b}.

    \begin{subclaim}
      Each $H_k^\beta$ is in $\cM | \alpha$ and the map
      $(\beta, k) \mapsto H_k^\beta$ is $\Sigma_1$-definable in
      $\cM|\alpha$ from parameters $r$ and $\bR^\cM$.
    \end{subclaim}
  
    Moreover, we also get an analogue of \cite[Claim 4.3]{St08b},
    stating that the positions $u$ from which Player I has a winning
    quasi-strategy in $G_x^\beta$ are precisely the $(\beta,x)$-honest
    positions.

    \begin{subclaim}
      For all positions $u$ in $G_x^\beta$ and all natural numbers
      $k$, $B_k^\beta(x,u)$ if, and only if, $H_k^\beta(x,u)$.
     \end{subclaim}
     \begin{proof}
       It is easy to see that $H_k^\beta(x,u)$ implies
       $B_k^\beta(x,u)$ as Player I can win from an $(\beta,x)$-honest
       position $u$ by continuing to play according to the true model
       $\cM|\beta$. For the other implication, let $\sigma$ be a
       winning quasi-strategy for Player I from a position $u$ in
       $G_x^\beta$. Recall that $\bR^{\cM}$ is countable in $V$ and
       consider a complete run
       $((i_n,x_{2n},\gamma_n,x_{2n+1}) \st n<\omega)$ of $G_x^\beta$
       according to $\sigma$ such that
       $\{x_i \st i \in \omega\} = \bR^\cM$. Moreover, consider the
       canonical model $\cN$ associated to this run of $G_x^\beta$ as
       above. We need to show that $\cN$ is an initial segment of
       $\cM|\beta$.

       This part of the proof uses a comparison argument. Recall that
       the standard proof of the comparison lemma (see for example
       Theorem 3.11 in \cite{St10}) uses a reflection argument to a
       small elementary substructure and hence $\dc$. But $\cM|\beta$
       and hence $\cN$ is $(\omega_1+1)$-iterable in $V$, so we can
       perform the comparison in $\HOD_{x,\Sigma,\Sigma^\prime}$,
       where $x$ is a real coding $\cM|\beta$ and $\cN$, and $\Sigma$
       and $\Sigma^\prime$ are iteration strategies for $\cM|\beta$
       and $\cN$ respectively. Similar as before, $\Sigma$ and
       $\Sigma^\prime$ are amenable to $\HOD_{x,\Sigma,\Sigma^\prime}$
       and their (canonically well-ordered) fragments
       $\Sigma \cap \HOD_{x,\Sigma,\Sigma^\prime}$ and
       $\Sigma^\prime \cap \HOD_{x,\Sigma,\Sigma^\prime}$ witness
       iterability in $\HOD_{x,\Sigma,\Sigma^\prime}$, which is a
       model of the Axiom of Choice. So there is no further assumption
       on $\cM$ needed and we obtain that $\cN$ is an initial segment
       of $\cM|\beta$, in fact that $\cN = \cM | \xi$, where $\xi$ is
       least such that $\cM|\xi \vDash \varphi(x,r,\bR)$, as in the
       proof of \cite[Claim 4.3]{St08b}.
     \end{proof}
     
     Now let $(\leq_i^\beta)_i$ be the quasi-scale on $B^\beta$
     constructed from the closed game representation as in 2.6 in
     \cite{Ke84} using the \emph{fake sup}, \emph{min}, and \emph{fake
       inf} method. Then $(\beta, i) \mapsto \leq_i^\beta$ is
     $\Sigma_1$ definable over $\cM|\alpha$ with parameters $r$ and
     $\bR^\cM$ as well, as desired.
    \end{proof}

    Using Claim \ref{cl:quasi-scale} together with Lemma
    \ref{lem:Kechris} we can now show the following claim.
    
  \begin{claim}\label{cl:basisresult}
    Every nonempty set of reals $B$ in $\cM | \alpha$ which is
    $\Sigma_1$-definable in $\cM | \alpha$ from a real parameter $r$
    and the parameter $\bR^\cM$, contains an element $x$ which is
    first-order definable in $\cM|\alpha$ from $r$ and $\bR^\cM$.
  \end{claim}
  
  \begin{proof}
    We will use Lemma \ref{lem:Kechris} to pick an element out of a
    set of reals $B$ in a definable way using a quasi-scale on
    $B$. Recall that $\ac_{\omega, \bR^\cM}$ holds in $\cM$ as a
    consequence of $\ad$. To obtain Claim \ref{cl:basisresult}, apply
    Lemma \ref{lem:Kechris} inside $\cM|\alpha$ to a nonempty set
    $B \subseteq \bR^\cM$ which is $\Sigma_1$-definable in
    $\cM|\alpha$ from some real parameter $r$ and the parameter
    $\bR^\cM$, the quasi-scale on $B$ obtained in Claim
    \ref{cl:quasi-scale}, and the pointclass $\Gamma$ of all sets
    which are first-order definable in $\cM|\alpha$ from the
    parameters $r$ and $\bR^\cM$.
  \end{proof}
  
  Claim \ref{cl:basisresult} now implies Lemma
  \ref{lem:uniformization}. Suppose $P$ is as in Lemma
  \ref{lem:uniformization}. By Claim \ref{cl:definability} we can in
  addition assume that $P$ is $\Sigma_1$-definable in $\cM | \alpha$
  from a parameter $r \in \bR^\cM$ and the parameter $\bR^\cM$. We can
  define a uniformizing function $F$ as follows. If for a real $x$,
  $\neg \exists y P(x,y)$, let $F(x) = x$. Otherwise, let $F(x)$ be
  the least (with respect to a fixed enumeration of first-order
  formulae) real $z$ which is first-order definable from $x,r$, and
  $\bR^\cM$ in $\cM | \alpha$ such that $P(x,z)$. Then $F \in \cM$ is
  the desired uniformization.
  \end{proof}

  This finishes the proof of Theorem \ref{thm:ADimpliesDCctblemice}.
\end{proof}

\section{Uncountable mice with $\dc_{\bR^\cM}$ in the
  background}\label{sec:noDCinthebackground}
  
In this section we argue that instead of working with countable
premice $\cM$ we can work in a background universe which is a model of
$\dc_{\bR^\cM}$, i.e. we derive Theorem \ref{thm:ADimpliesDC} as a
corollary of Theorem \ref{thm:ADimpliesDCctblemice}.

\begin{proof}[Proof of Theorem \ref{thm:ADimpliesDC}]
  Let $\cM$ be a passive, countably iterable $\bR^\cM$-premouse such
  that $\cM \vDash \ad$. Using $\dc_{\bR^\cM}$ in $V$, we can by the
  standard proof of the L\"owenheim-Skolem Theorem consider a
  countable elementary substructure $\cN$ of $\cM$. Then $\cN$ is an
  $(\omega_1+1)$-iterable $\bR^{\cN}$-premouse and we can apply
  Theorem \ref{thm:ADimpliesDCctblemice} to $\cN$. This yields
  $\cN \vDash \dc$ and hence $\cM \vDash \dc$.
\end{proof}

Finally, note that the statements in Theorem \ref{thm:ADimpliesDC} and
Theorem \ref{thm:ADimpliesDCctblemice} are in fact equivalent by the
following argument. Let $\cM$ be a countable, passive,
$(\omega_1+1)$-iterable $\bR^\cM$-premouse such that $\cM \vDash
\ad$. Let $\Sigma$ be an $(\omega_1+1)$-iteration strategy for $\cM$
and $x_\cM$ be a real coding $\cM$. Now apply Theorem
\ref{thm:ADimpliesDC} inside $\HOD_{x_\cM,\Sigma}$, which is a model
of the Axiom of Choice. 

Using that for any countable set of reals $A$ the Woodin cardinals in
$M_n(A)$ are countable in $V$, we obtain the following corollary.

\begin{corollary}[$\zf$]
  Let $n \leq \omega$ and let $A \in \ctbleset$. Suppose that
  $M_n^\sharp(A)$ exists and is $(\omega_1+1)$-iterable. Moreover,
  suppose that $M_n(A) \cap \bR = A$ and $M_n(A) \vDash \ad$. Then
  $M_n(A) \vDash \dc$.
\end{corollary}

\bibliographystyle{abstract}
\bibliography{References}

\end{document}